\documentclass[12pt]{article}

\usepackage{amsmath,amssymb}
\usepackage{latexsym}
\usepackage{graphicx}
\usepackage{bm}
\usepackage{hyperref}

\newenvironment{proof}{\begin{trivlist}
                       \item[]{\bf Proof.}
                       \hspace{0cm}}{\hfill $\Box$
                       \end{trivlist}}

\date{}

\begin{document}


\centerline{}

\centerline{}

\centerline {\Large{\bf Representation of vector fields}}

\centerline{}

\centerline{\bf {A. G. Ramm$\dag$\footnotemark[3]}}

\centerline{}

\centerline{$\dag$Mathematics Department, Kansas State University,}

\centerline{Manhattan, KS 66506-2602, USA}

\renewcommand{\thefootnote}{\fnsymbol{footnote}}
\footnotetext[3]{Corresponding author. Email: ramm@math.ksu.edu}

\newtheorem{Theorem}{\quad Theorem}[section]

\newtheorem{Definition}[Theorem]{\quad Definition}

\newtheorem{Corollary}[Theorem]{\quad Corollary}

\newtheorem{Lemma}[Theorem]{\quad Lemma}

\newtheorem{Example}[Theorem]{\quad Example}

\newtheorem{Remark}[Theorem]{\quad Remark}

\newtheorem{thm}{Theorem}[section]
\newtheorem{cor}[section]{Corollary}
\newtheorem{lem}[section]{Lemma}
\newtheorem{dfn}[section]{Definition}
\newtheorem{rem}[section]{Remark}

\newcommand{\bee}{\begin{equation*}}
\newcommand{\eee}{\end{equation*}}
\newcommand{\be}{\begin{equation}}
\newcommand{\ee}{\end{equation}}
\newcommand{\ba}{\begin{align}}
\newcommand{\ea}{\end{align}}
\newcommand{\pn}{\par\noindent}
\newcommand{\RRR}{\mathbb{R}^3}

\begin{abstract} \noindent
A simple proof is given for the explicit formula which allows one to recover a $C^2-$smooth 
vector field $A=A(x)$ in $\mathbb{R}^3$, decaying at infinity, from the knowledge of its $\nabla \times A$
and $\nabla \cdot A$. The representation of $A$ as a sum of the gradient field and a divergence-free
vector fields is derived from this formula. Similar results are obtained for a vector field in a 
bounded   $C^2-$smooth domain.
\end{abstract}



\section{Introduction}

In fluid mechanics and electrodynamics one is often interested in the following questions:
\begin{itemize}
\item[Q1.] Let $A(x), x \in \mathbb{R}^3$, be a twice differentiable in $\mathbb{R}^3$ vector field vanishing at infinity together with its two derivatives. Given $\nabla \times A$ and $\nabla \cdot A$, can one recover $A(x)$ uniquely? Can one give an explicit formula for $A(x)$?
\item[Q2.] Can one find a scalar field $u = u(x)$ and a divergence-free vector field $B(x)$, $\nabla \cdot B = 0$, such that
\be\label{eq1}
A = \nabla u + B, \qquad \int_{\mathbb{R}^3}\nabla u \cdot B dx = 0.
\ee
\end{itemize}

These questions were widely discussed in the literature, for example, in \cite{1} - \cite{3}. Our aim is to give a simple answer to these questions. By $H^m(\RRR)$, $H^m(D)$, the usual Sobolev spaces are denoted, $H^m(D, w(x))$ is the weighted Sobolev space, where $w=w(x)>0$ is the weight function.

\section{Answer to question Q1.}
Denote $\nabla \times A := a$, $\nabla \cdot A := f$. Then $\nabla \times \nabla \times A = \nabla \times a$. It is well known that
\be\label{eq2}
-\nabla^2A = \nabla \times \nabla \times A - \nabla \, \nabla \cdot A.
\ee
Thus,
\be\label{eq3}
-\nabla^2 A = \nabla \times a - \nabla f.
\ee
Let $g(x,y) := \frac{1}{4\pi|x - y|}$. Then
\be\label{eq4}
-\Delta g(x,y) = \delta(x - y),
\ee
where $\delta(x)$ is the delta function. Thus, from \eqref{eq3} one gets
\be\label{eq5}
A(x) = \int_{\RRR}g(x,y) \nabla \times a dy - \int_{\RRR}g(x,y) \nabla f dy.
\ee

{\em This formula gives an analytical representation of $A(x)$ in terms of $a = \nabla \times A$ and $f = \nabla \cdot A$.}

To prove {\em the uniqueness of this representation,} assume that there are two different vector fields $A$ and $F$ that have the same $a = \nabla \times A = \nabla \times F$ and $f = \nabla \cdot A = \nabla \cdot F$. Then, by formula \eqref{eq2}, one has
\be\label{eq6}
-\nabla^2(A - F) = 0.
\ee

Therefore $A - F$ is a harmonic function in $\RRR$ which vanishes at infinity. By the maximum principle such a function is equal to zero identically.
 
{\em  Thus, $A(x)$ is uniquely determined in $\RRR$ by formula \eqref{eq5} if $\nabla \times A$ and $\nabla \cdot A$ are known and if $A$ vanishes at infinity.} \hfill $\Box$

\section{Answer to question Q2.}
Formula \eqref{eq5} can be written as
\be\label{eq7}
A(x) = \nabla \times \int_{\RRR}g(x,y) a(y) dy - \nabla \int_{\RRR}g(x,y) f(y)dy,
\ee
provided that $a(y)$ and $f(y)$ decay at infinity sufficiently fast, for example, if 
$$|A(x)| + |\partial A(x)| + |\partial^2 A(x)| \leq c(1 + |x|)^{-\gamma}, \gamma > 3,$$ 
where $\partial$ is an arbitrary first order derivative, so that the following integrations by parts can be justified:
\be\label{eq8}
\nabla \times \int_{\RRR}g(x,y) a(y) dy = - \int_{\RRR}\left[ \nabla_y g(x,y), a(y)  \right]dy = \int_{\RRR}g(x,y)\nabla \times a(y) dy,
\ee
\be\label{eq9}
-\nabla \int_{\RRR}g(x,y) f(y) dy = \int_{\RRR}\nabla_y g(x,y) f(y) dy = - \int_{\RRR} g(x,y) \nabla f(y) dy.
\ee
One may also assume that $A \in H^2(\RRR, 1 + |x|^\gamma), \gamma > 2,$ in order to justify formulas \eqref{eq1}, \eqref{eq5}, \eqref{eq7}.

It follows from \eqref{eq1} and \eqref{eq7} that
\be\label{eq10}
u(x) = -\int_{\RRR} g(x,y) f(y) dy, \qquad B(x) = \nabla \times \int_{\RRR} g(x,y) a(y) dy.
\ee

To check the second formula \eqref{eq1}, it is sufficient to check that
\be \label{eq11}
	\int_{\RRR} \nabla u \cdot \nabla \times p dx=0,
\ee

provided that $u=u(x)$ and $p=p(x)$ decay at infinity sufficiently fast. In our case
$u$ is defined by formula \eqref{eq10} and $p(x)=\int_{\RRR} g(x,y) a(y)dy$.

Formula \eqref{eq11} can be verified by a direct calculation. Let $\frac{\partial u}{\partial x_j}:= u_{,j}$ and denote by $e_{jmq}$ the antisymmetric unit tensor: $e_{123}=1, e_{jmq}=\left\{
\begin{array}{ll}
	1  & \text{ if $jmq$ is even},\\
	-1 & \text{ if $jmq$ is odd}.
\end{array}
\right.$
The triple $jmq$ is called even if by an even number of transpositions it can be reduced to the triple 123. An odd triple $jmq$ is the one that is not even. A transposition is the change of the order of two neighboring indices.

The vector product can be written with the help of $e_{jmq}$ as follows: 
$$A \times B = [A,B]=e_{jmq} A_m B_q.$$ 
Here and below summation is understood over the repeated indices. For example, $(\nabla \times p)_j=e_{jmq} p_{q,m}$. With these notations one has
\be \label{eq12}
    \int_{\RRR} \nabla u\cdot \nabla\times p dx=e_{jmq}\int_{\RRR} u_{,j}p_{q,m}dx=-e_{jmq} \int_{\RRR} u p_{q,mj}dx=0,
\ee
because $e_{jmq}p_{q,mj}=0$.

Let us summarize the results.

\begin{thm} \label{thm1}
    Assume that a vector field $A(x)\in H_{loc}^2(\RRR)$ decays at infinity sufficiently fast, for example, $A(x)\in H^2(\RRR, 1+|x|^\gamma), \gamma >2$. Then, given $a:=\nabla\times A$ and $f:=\nabla\cdot A$ in $\RRR$ one can uniquely recover $A$ by formula \eqref{eq5}.

    Moreover, one can uniquely represent $A(x)$ by formula \eqref{eq1}, where $u$ and $B$ are uniquely defined by formula \eqref{eq10}.
\end{thm}
\begin{thm} \label{thm2}
    Assume that $D \subset \RRR$ is a bounded domain with $C^2$-smooth boundary $S$, $A(x) \in H^2(D), a(x):=\nabla\times A(x), f:=\nabla\cdot A(x)$ and $\phi(s)=A|_{s \in S}$ are known. Then $A(x)$ is uniquely recovered by solving the Dirichlet problem
    \be \label{eq13}
        -\nabla^2 A=\nabla\times a(x)-\nabla f(x) \text{ in }D, \qquad A|_S=\phi(s).
    \ee
\end{thm}
\begin{proof}
    Theorem \ref{thm1} is already proved.

    To prove Theorem \ref{thm2} one reduces it to solving problem \eqref{eq13}. Existence and uniqueness of the solution to the Dirichlet problem \eqref{eq13} are known, so Theorem \ref{thm2} is proved.
\end{proof}
\begin{Remark}
    It follows from formula \eqref{eq7} that if $f=\nabla\cdot A=0$, then $A=B=\nabla\times \int_{\mathbb{R}^3} g(x,y)a(y)dy$, and if $a=\nabla\times A=0$, then $A=\nabla u$, where $u$ is defined in formula \eqref{eq10}.
\end{Remark}
\begin{Remark}
    Under the assumption $\gamma>3$, vector field $A(x)$ decays at infinity so that formulas \eqref{eq1}, \eqref{eq5}, and \eqref{eq10} are valid.
\end{Remark}
Let us estimate, for example, an integral of the type \eqref{eq9} assuming that $|\nabla f|\le \frac{c}{(1+|x|)^\gamma}, \gamma > 3$. Let $|x|=r, |y|=\rho$, $\theta$ be the angle between $x$ and $y$, and $x$ is directed along $y_3$ axis. Then one has
\begin{align}
    I_1: &=\int_{\RRR} \frac{dy}{|x-y|(1+|y|)^\gamma} \\
         &=2\pi \int_0^\infty \frac{dr r^2}{(1+r)^\gamma}\int_{-1}^1 \frac{ds}{(r^2-2r\rho s +\rho^2)^{1/2}} \\
         &=\frac{\pi}{\rho} \int_0^\infty \frac{dr r}{(1+r)^\gamma} (r+\rho -|r-\rho|) \\
         &=\frac{\pi}{\rho} \left(2\int_0^\rho \frac{dr r^2}{(1+r)^\gamma} + \int_\rho^\infty \frac{dr r 2 \rho}{(1+r)^\gamma} \right)\\
         &\le \frac{\pi}{\rho}\left( 2\left.\frac{(1+r)^{-\gamma+3}}{-\gamma+3}\right|_0^\rho + 2\rho\left.\frac{(1+r)^{-\gamma+2}}{-\gamma+2}\right|_\rho^\infty \right) \\
         &\le \frac{2\pi}{\rho(\gamma-3)}+\frac{2\pi}{\gamma-2}\frac{1}{\rho^{\gamma-2}}.
\end{align}
If $A \in H^2(\RRR, (1+|x|)^\gamma), \gamma>2$, let us estimate, for example, the following integral:
\begin{align}
    I_2^2 &:= \left(\int_{\RRR}\frac{1}{|x-y|}|\nabla\times a|dy \right)^2 \\
          &\le \int_{\RRR} \frac{dy}{|x-y|^2(1+|y|)^\gamma} \int_{\RRR}|\nabla\times a|^2(1+|y|)^\gamma dy \\
          &\le c\int_{\RRR} \frac{dy}{|x-y|^2(1+|y|)^\gamma} \\
          &\le 2\pi c\int_0^\infty \frac{dr r^2}{(1+r)^\gamma}\int_{-1}^1 \frac{ds}{r^2-2r\rho s +\rho^2} \\
          &=\frac{\pi c}{\rho} \int_0^\infty \frac{dr r}{(1+r)^\gamma} \ln \frac{c_1 \ln\rho}{\rho}, \quad \rho>1.
\end{align}
By $c, c_1>0$ estimation constants are denoted.


\end{document}